\documentclass[12pt]{amsart}
\usepackage{latexsym, amsmath, amssymb, amsthm, mathrsfs,wasysym}
\usepackage[alwaysadjust]{paralist}
\usepackage{caption}
\usepackage{subcaption}
\usepackage[T1]{fontenc}
\usepackage{lmodern}
\usepackage[backref,colorlinks=true,linkcolor=blue,urlcolor=blue, citecolor=blue]%
  {hyperref}
\usepackage[shortalphabetic]{amsrefs}
\usepackage[all]{xy}
\usepackage[T1]{fontenc}
\usepackage{mathptmx}
\usepackage{extarrows}
\usepackage{microtype}
\usepackage{framed}
\usepackage{tikz}
\usepackage{tikz-cd}
\usepackage{graphicx}
\usepackage[alwaysadjust]{paralist}
\makeatletter
\providecommand\@enum@widestlabel{7}
\makeatother

\usepackage[centering, includeheadfoot, hmargin=1in, vmargin=1in,
  headheight=30.4pt]{geometry}
\newtheorem{lemma}{Lemma}[section]
\newtheorem{theorem}[lemma]{Theorem}
\newtheorem{corollary}[lemma]{Corollary}
\newtheorem{proposition}[lemma]{Proposition}
\newtheorem{conjecture}[lemma]{Conjecture}
\theoremstyle{definition}

\newtheorem{remark}[lemma]{Remark}

\newtheorem{claim}[lemma]{Claim}

\renewcommand{\theequation}%
{\arabic{section}.\arabic{lemma}.\arabic{equation}}

\newcommand{\CC}{\ensuremath{\mathbb{C}}} 
 
\newcommand{\PP}{\ensuremath{\mathbb{P}}} 
\newcommand{\QQ}{\ensuremath{\mathbb{Q}}} 
\newcommand{\RR}{\ensuremath{\mathbb{R}}} 
\newcommand{\ZZ}{\ensuremath{\mathbb{Z}}} 
 
\newcommand{\sI}{\ensuremath{\kern -1pt \mathscr{I}\kern -2pt}} 
\newcommand{\sJ}{\ensuremath{\kern -2pt \mathscr{J}\kern -2pt}} 

\newcommand{\sO}{\ensuremath{\mathscr{O}}}

\renewcommand{\geq}{\geqslant}
\renewcommand{\leq}{\leqslant}

\DeclareMathOperator{\mult}{mult}

\DeclareMathOperator{\Pic}{Pic}

\newcommand{\deq}{\ensuremath{\stackrel{\textrm{def}}{=}}}

\newcommand{\Bplus}{\ensuremath{\textbf{\textup{B}}_{+} }}

\newcommand{\Bstable}{\ensuremath{\textbf{\textup{B}} }}

\input xy
\xyoption{all}

\begin{document}

\title{Seshadri constants of indecomposable polarized abelian varieties}

\author[V.~Lozovanu]{Victor Lozovanu}

\address{Victor Lozovanu -- Current Address: Universit\'a degli Studi di Genova, 
\newline  
\hspace*{2.74in} Dipartimento di Matematica, 
\newline  
\hspace*{2.74in} Via Dodecaneso 35, 16146, Genova, Italy.\newline
\hspace*{2.74in} \textit{Email address}: \href{lozovanu@dima.unige.it}{\nolinkurl{lozovanu@dima.unige.it}}}

\address{\hspace*{1.9in} Address: Leibniz Universit\"{a}t Hannover,  \newline
 \hspace*{2.74in} Institut f\"{u}r Algebraische Geometrie, 
 \newline
 \hspace*{2.74in}
 Welfengarten 1, 30167, Hannover, Germany.
 \newline
\hspace*{2.74in} \textit{Email address}: \href{lozovanu@math.uni-hannover.de}{\nolinkurl{lozovanu@math.uni-hannover.de}}}

\maketitle

\begin{abstract}
	The goal of this note is to study a conjectural picture on lower bounds of Seshadri constants of indecomposable polarized abelian varieties. This is inspired by some ideas of Debarre on the subject and the author's previous work on syzygies of abelian threefolds using the convex geometry of Newton--Okounkov bodies.
\end{abstract}

\section{Introduction}
Let $X$ be a smooth complex projective variety and $L$ an ample line bundle on $X$. To measure the positivity of $L$ at a point $x\in X$, Demailly defines the \textit{Seshadri constant} of $L$ at $x$ to be
\[
\epsilon(L;x) \ \deq \ \inf_{x\in C\subseteq X}\Big\{\frac{(L\cdot C)}{\mult_0(C)}\Big\} \ ,
\]
where the infimum is taken over all reduced and irreducible curves $C$ on $X$ containing $x$. This invariant encodes numerically all the "minimal curves" through $x$. It can also be seen from an infinitesimal perspective  (see \cite[Section~5]{PAG} for a nice introduction in the field). 

Two lines of research became prominent in the area.  First, differentiation techniques, used in \cite{EL93, EKL95, N96}, lead to strong lower bounds on Seshadri constant, when $x\in X$ is very general. Second, it turns out that these invariants are connected to other fields of geometry. They appear in K\"ahler geometry \cite{Ny15,Ny18}, diophantine approximation problems \cite{MR15}, convex geometry \cite{KL17} and positivity issues of abelian varieties \cite{LPP11, Loz18, KL19, Loz20}.

The main goal of this note is to give credence to a conjectural picture on lower bounds of Seshadri constants on abelian manifolds. So, let $(A,L)$ be an abelian variety of dimension $g\geq 1$. Due to the group structure on $A$, the Seshadri constant $\epsilon(L)$ of $L$ doesn't depend on the base point. Thus, differentiation techniques from \cite{ELN94} lead Nakamaye \cite{N96} to prove that
\[
\epsilon(L) \ \geq \ 1.
\]
The equality holds if and only if $(A,L)$ is a product of an abelian subvariety and an elliptic curve. Consequently, Nakamaye asked whether indecomposable pairs have much larger Seshadri constants.

With this in hand, we propose and give some credence to the following conjecture:
\begin{conjecture}[Debarre]\label{conj:main}
	Let $(A,L)$ be a $g$-dimensional indecomposable polarized abelian variety with $\epsilon(L)<2$, then  $(A,L)$ is isomorphic to the Jacobian $(J_C,\Theta_C)$ of a smooth  curve $C$ of genus $g$. Moreover, when $g\geq 4$ the curve $C$ has to be hyper-elliptic.
\end{conjecture}
Inspired by a classical conjecture of van Geemen and van der Geer \cite{GG86}, Debarre  introduces this statement in \cite{D04} for theta-divisors. He shows that for the Jacobian of hyper-elliptic curves one has $\epsilon(\Theta_C)=\frac{2g}{g+1}$. Moreover, Debarre's initial conjecture on theta divisors holds in small dimensions. When $g=3$ it was settled by Bauer-Szemberg in \cite{BS01} and when $g=4$ it follows from the ideas developed by Izadi in \cite{I95}. In higher dimensions it remains an open question.

The main goal of this paper is to show that Congecture \ref{conj:main} in this more general form follows from the original Debarre's conjecture for theta divisors.
\begin{theorem}
	\label{thm:main}
If Conjecture~\ref{conj:main} holds for any irreducible principal polarized abelian variety of dimension $\leq g$, then it holds for any indecomposable polarized abelian variety of dimension $\leq g$.
\end{theorem}
Combining with the work in \cite{BS01} and \cite{I95}, this statement provides the first non-trivial bound on the Seshadri constant for indecomposable polarized abelian manifolds of small dimension. Thus it provides a complete answer to the initial question by Nakamaye in this case.
\begin{corollary}\label{cor:main}
Conjecture \ref{conj:main} holds in dimension $g=1,2,3,4$.
\end{corollary}
The main reason behind Theorem \ref{thm:main} is the ``minimality principle'' for polarizations on abelian varieties. More precisely, \cite[Proposition~4.1.2]{BL04} yields that for any ample line bundle $L$ on $A$, there is an \'etale map of degree $k\geq 1$:
\[
f: (A,L)\ \longrightarrow \ (A_L,\Theta_L) \ ,
\]
such that $f^*(\Theta_L)=L$ and $\Theta_L$ is a theta divisor on $A_L$. Consequently, one gets $\epsilon(L)\geq  \epsilon(\Theta_L)$. 

Assuming $\epsilon(L)<2$, then Debarre's initial conjecture for principal polarizations would imply Theorem \ref{thm:main} as long as we can get contradictions when the pair $(A_L,\Theta_L)$ is either decomposable, or the Jacobian of a hyperelliptic curve of genus $g\geq 3$ or the Jacobian of a non-hyperelliptic curve of genus $g=3$.

In the first case the map $f$ allows us to find decomposable effective divisors in the linear series $|L|$. So, B\'ezout's theorem and an inductive argument will do the trick.

When $(A_L,\Theta_L)\simeq (J_C,\Theta_C)$ for a hyperelliptic curve $C$ of genus $g$, we study certain Newton-Okounov polygons of $L$ on the surface $f^*(C-C)$. These convex sets have large areas, equal to $\frac{1}{2}k\cdot g(g-1)$, due to \cite[Theorem C]{LM09}. So, we would get a contradiction once we show that geometry imposes strong conditions on the shape of these convex sets, forcing their areas to be actually small. For this, we use arithmetics and geometry to connect the Seshadri curve of $L$ to the Seshadri curve of $\Theta_C$ via the map $f$. Then, knowing strong properties on some of the Zariski chambers decomposition of the pseudo-effective cone of $\textup{Sym}^2(C)$, that desingularizes $C-C$, we find strong upper bounds on the shape of our polygons by making use of \cite[Theorem 6.4]{LM09}. 

The case when $(A_L,\Theta_L)\simeq (J_C,\Theta_C)$ for a non-hyperelliptic curve $C$ of genus $g=3$ remains the most complicated to treat. The most important issue  is that the connection between the Seshadri curve of $L$ and the one for $\Theta_C$ via the map $f$ is not that strong, and geometry will get us only some lower bounds on its numerics. To turn the table, we get some inspiration from the author's work \cite{Loz18} on syzygies of abelian threefolds. The goal is to look at the blow-up $\pi:\overline{A}\rightarrow A$ of the origin $0\in A$, and study the convex geometry in $\RR^3$ of an infinitesimal Newton-Okounkov body of $L$. With respect to previous work in the literature, for example \cite{KL17}, the difference here is that we construct this convex set by using a non-standard flags, given partially by the tangent cone at the origin of the surface $C-C\subseteq J_C$, which is a quartic in the exceptional divisor $E\simeq \PP^2$. This choice alleviates some of the issues when using differentiation techniques from \cite{CN14} and \cite{Loz18} to study the behaviour of asymptotic multiplicities, as our flags will turn out to be local complete intersections. As a consequence we get strong upper bounds on the vertical slices of our convex set, due to the low value of the Seshadri constant, but whose volume equal to $\frac{L^3}{6}=k$ by \cite{LM09} is quite large. This leads to a contradiction with one exception, when $k=2$ and the Seshadri curve of $L$ has small numerics. Since the geometry of ètale covers of degree two are much simpler to understand, we are able to get a contradiction even in this case by relying on the same methods applied to the hyperelliptic case, i.e. considering a Newton-Okounkov polygon on the surface $f^*(C-C)$.

The current bounds play an important role in the author's recent work on singularities of irreducible theta divisors in any dimension \cite{Loz20}.

\subsection*{Acknowledgements} Special thanks for the amazing support to all the members of the Institut f\"ur Algebraische Geometrie at Leibniz Universit\"at Hannover. The author is grateful to V\'ictor Gonzalez-Alonso, whose expertise on abelian varieties played an important role during this project. Also, many thanks are due to M. Fulger and A. K\"uronya for helpful discussions about some of the ideas in this article.

\section{Notations and preliminaries}
\subsection{Notations}
In this article we work over the complex numbers $\CC$. A pair $(A,L)$ stands for a $g$-dimensional abelian variety together with an ample polarization (line bundle) $L$. The pair $(A,L)$ is said to be \textit{indecomposable} if it's not isomorphic to the product of two polarized abelian varieties.

Most of the time our problems are local. So, we will be translating them to the blow-up af the origin of $A$. We fix some notation in this infinitesimal setting. We set $\pi_A:\overline{A}\rightarrow A$ to be the blow-up of the origin $0\in A$ with the exceptional divisor $E_A\simeq\PP^{g-1}$. We denote by
\[
B_t \ \deq \ \pi_A^*(L)-tE_A, \textup{ for any } t\geq 0 \ .
\]
By \cite[Proposition~5.1.5]{PAG}, then the Seshadri constant can be defined as follows
\[
\epsilon(L)\ = \ \max\{t\geq 0\ | \ B_t\textup{ is nef}\}\ .
\]
Moreover, we define the \textit{infinitesimal width} of $L$ as follows
\[
\mu(L) \ \deq \ \max\{t\geq 0\ | \ B_t\textup{ is pseudo-effective}\} \ .
\]
This is the maximum multiplicity at the origin of a $\QQ$-effective divisor in the class of $L$. When there is no confusion we denote by $\epsilon=\epsilon(L)$ and $\mu=\mu(L)$.

\subsection{Augmented base locus}
Let $L$ be a big line bundle on a smooth projective variety $X$. It's natural to consider the stable base locus $\Bstable(L)$, where all the effective $\QQ$-divisors in $L$ vanish. Unfortunately, this locus does not behave well, neither in its numerical class nor inside the Néron-Severi space of $X$. So, the authors of \cite{ELMNP06} introduced the \textit{augmented base locus} of $L$ to be 
\[
\Bplus(L) \ \deq \ \Bstable(L-\frac{1}{m}A) \ ,
\]
for an ample line bundle $A$ on $X$ and some positive integer $m\gg 0$. This locus doesn't depend on the choice of $A$ ample. It is a numerical invariant of $L$ and can be extended to any real class $\xi\in N^1(X)_{\RR}$. 

For example, when $X=\overline{A}$ is the blow-up of the abelian variety $A$ at the origin, as in the previous subsection, then the function $t\mapsto \Bplus(B_t)$ is increasing with respect to the inclusion and remains constant outside of at most countably many jumps by \cite[Lemma 1.3]{Loz18}.

Finally, to measure how bad the class of $L$ vanishes along a subvariety $V\subseteq \Bplus(L)$ we use the multiplicity. So, the \textit{asymptotic multiplicity} of $L$ along $V$ is defined to be
\[
\mult_V(||L||) \ \deq \ \lim_{m\rightarrow \infty}\ \frac{\mult_V(|mL|)}{m} \ .
\]
This definition can be extended to any real class in $N^1(X)_{\RR}$. When $X=\overline{A}$ is the above blow-up, then the function $t\mapsto \mult_V(||B_t||)$ is continuous, increasing and convex by \cite[Lemma 1.5]{Loz18}. Moreover, it has a large slope due to differentiation techniques, see \cite[Proposition 3.3]{Loz18}.

An easy numerical algorithm to check when subvarieties are contained in the augmented base locus is discussed in the following lemma:
\begin{lemma}\label{lem:curvesurface}
Let $L$ be a big line bundle on a smooth projective variety $X$. Let $\xi$ be a $\QQ$-Cartier divisor on an irreducible subvariety $Y\subseteq X$ of dimension $k$ that defines a nef class. If $(L\cdot \xi^{k-1}) <  0$, then $Y\subseteq \Bplus(L)$.
\end{lemma}
\begin{proof}
We show that a general point $x\in Y\setminus \textup{Sing}(Y)$ is contained in $\Bplus(L)$. Let $\pi:Z\rightarrow X$ be a resolution of singularities of $Y$ that is an isomorphism over $X\setminus \textup{Sing}(Y)$ (see \cite[Theorem 4.1.3]{PAG}). Let $\overline{Y}$ be the proper transform of $Y$ through the map $\pi$, which is a smooth subvariety of $Z$.

By \cite[Proposition 2.3]{BBP13} we know that $\pi^{-1}(\Bplus(L))=\Bplus(\pi^*(L))\cap \textup{Exc}(\pi)$. So, by our assumptions above it suffices to show that a general point $z\in \overline{Y}$ is also contained in $\Bplus(\pi^*(L))$.

For this note first, that $\pi^*(\xi)$ remains a $\QQ$-Cartier nef class and $(\pi^*(L)\cdot \pi^*(\xi))<0$. Now, the classical theorem of Kleiman (\cite[Theorem 1.4.23]{PAG}) implies that inside the Néron-Severi space $N^1(\overline{Y})_{\RR}$ the class of $\pi^*(\xi)$  is a limit of ample $\QQ$-Cartier classes. In particular, one can find an ample class $\eta\in N^1(\overline{Y})_{\QQ}$ such that $(\pi^*(L)\cdot \eta^{k-1})<0$. 

Let $z\in \overline{Y}$ be a general point. By Serre and Bertini's theorems, there is an integer $m\gg 0$ such that $m\cdot \eta$ is a Cartier divisor on $\overline{Y}$ and there exists divisors $D_1,\ldots ,D_{k-1}\in |m\eta|$ such that the scheme-theoretical intersection $C=D_1\cap\ldots \cap D_{k-1}$ is an irreducible curve through $z$. Since
\[
(\pi^*(L)\cdot \eta^{k-1})\ = \ \frac{1}{m^{k-1}}(\pi^*(L)\cdot D_1\cdot\ldots\cdot D_{k-1})\ = \ \frac{1}{m^{k-1}}(\pi^*(L)\cdot C)
\]
then $(\pi^*(L)\cdot C)<0$. Thus $C\subseteq \Bstable(\pi^*(L))\subseteq \Bplus(\pi^*(L))$ and $z\in \Bplus(\pi^*(L))$.

\end{proof}

\subsection{Newton-Okounkov bodies on surfaces} As pointed out in the introduction an important role in the proof of Theorem \ref{thm:main} is knowing how to describe upper bounds on the Newton-Okounkov polygons for line bundles on algebraic surfaces. 

Let $S$ be a smooth projective surface. Consider $L$ a big line bundle on $S$, $Q\subseteq S$ an irreducible curve and $P\in Q$ a smooth point. With these data we can associate the Newton-Okounkov polygon 
\[
\Delta_{Q,P)}(L) \ \subseteq \ \RR^2.
\]
Following \cite[Theorem B]{KLM12} and \cite[Section 6.2]{LM09}, this set is described by considering the Zariski decomposition of the class $L-tQ=P_t+N_t$ for any $t\in\RR$, as long as it is big. Here $P_t$ is called the positive part and is a nef class. The divisor $N_t$ is the negative part and is an effective divisor satisfying certain properties, as explained in \cite[Theorem 2.3.19]{PAG}. If $\mu=\max\{t|L-tQ \textup{ is big}\}$ and let $\nu =\textup{ord}_Q(N_0)$. With these data, we know then 
\[
\Delta_{(Q,P)}(L) \ \subseteq \ [\nu,\mu]\times \RR_+ ,
\]
and the two boundary vertical lines touch the polygon. Finally, on $[\nu,\mu]$ our polygon $\Delta_{(Q,P)}(L)$ is the area between the graphs of the functions $\alpha(t)=\textup{ord}_P(N_t|_Q)$ and $\beta(t)=\alpha(t)+(P_t\cdot Q)$.

With these in hands, we have the following lemma.
\begin{lemma}\label{lem:boundary}
Under the above notation, let $P\in Q$ be a general point and suppose that the following conditions are satisfied:
\begin{enumerate}
\item $L$ is big and nef, $Q^2<0$, $(L\cdot Q)=0$ and set $\epsilon'=\textup{max}\{t\geq 0|L-tQ -\textup{nef}\}$.
\item Let $F$ be an irreducible curve contained in $\textup{Supp}(N_t)$ for all $0<t-\epsilon'\ll 1$, and suppose there is $\mu'\in (\epsilon',\mu]$ and a real number $a>0$, such that $\textup{mult}_F(||L-tQ||)\geq a(t-\epsilon')$ for any $t\in (\epsilon',\mu')$.
\end{enumerate}
 Then $\Delta_{(Q,P)}(L)$ is contained in the triangle with vertices: $(0,0),(\epsilon',-\epsilon'\cdot Q^2), (\epsilon'\cdot\frac{a(F\cdot Q)}{Q^2+a(F\cdot Q)}, 0)$. 
\end{lemma}
\begin{proof}
In $N^1(S)_{\RR}$ the classes of $Q$ and any curve appearing in $\textup{Supp}(N_t)$ for any $t\geq 0$ lie in different half-spaces defined by the hyperplane $(L-\epsilon'Q)^{\perp}$. Thus, $Q\nsubseteq \textup{Supp}(N_t)$ and $F\neq Q$. Moreover, $P\in C$ is a general point, so by \cite[Proposition 2.1]{KLM12} we have $\alpha(t)=0$ for any $t\in \RR$.

On the interval $[0,\epsilon']$ the class $L-tQ$ is nef and thus $\beta(t)=-tQ^2$. Hence, $\Delta_{(Q,P)}(L)\cap[0,\epsilon']\times \RR$ is exactly the triangle given by the vertices $(0,0),(\epsilon',0),(\epsilon',-\epsilon'Q^2)$.

On the interval $[\epsilon',\mu')$, let's denote by $m_t=\mult_F(||L-tQ||)$. Since the divisor $L-tQ-m_tF-P_t$ is $\QQ$-effective and $Q\nsubseteq \textup{Supp}(N_t)$, then we have the following sequence of inequalities
\[
\beta(t) \ = \ (P_t\cdot Q) \ \leq \ (L-tQ-m_tF\cdot Q) \ \leq (L-tQ\cdot Q)-a(t-\epsilon')(F\cdot Q) \ ,
\]
where the graph of the function given by the bound on the right is exactly the linear function connecting the points $(\epsilon',-\epsilon'\cdot Q^2)$ and $(\epsilon'\cdot\frac{a(F\cdot Q)}{Q^2+a(F\cdot Q)}, 0)$. Thus, $\Delta_{(Q,P)}(L)\cap[\epsilon', \mu']\times \RR$ is below this line and above the horizontal axis. Moreover, our convex set contains the point $(\epsilon',-\epsilon'\cdot Q^2)$, so convexity would force it to be below this line in its entirety.
\end{proof}

\section{\'Etale covers over hyperelliptic Jacobians}

In this section we study the behaviour of the Seshadri constant for an abelian polarized variety $(A,L)$ of dimension $g$ that is an étale cover over the Jacobian of a hyperelliptic curve of genus $g$.

We start by reminding the reader of the basic positivity results on such Jacobians and then continue with the proof of the main statement.

\subsection{Hyperelliptic Jacobians.}
We start with a short review of the classical infinitesimal picture for the Jacobian of a hyper-elliptic curve (see \cite{BL04}, \cite{BS01} and \cite{L96}).

Let $C$ be a hyper-elliptic curve  of genus $g\geq 3$ and let $(J_C,\Theta_C)$ be its associated Jacobian, that can be defined as the Picard group of the degree zero line bundles on $C$. The canonical divisor $K_C$ defines a morphism  $\phi_{K_C} : C \rightarrow \PP^{g-1}$ of degree two over its image, that is a smooth rational curve $Q\subseteq \PP^{g-1}$ of degree $g-1$. It's fibers define a natural involution $\sigma:C\rightarrow C$. 

There is also a natural embedding $S_C\deq C-C\subseteq J_C$. By \cite[Theorem~11.2.5]{BL04}, the surface $S_C$ is smooth everywhere, besides the origin, where $\mult_0(S_C)=g-1$, corresponding to the unique $\phi_{K_C}\in g_2^1$. We consider the  difference map
\[
\partial \ : \  C\times C\ \xlongrightarrow\  S_C=C-C, \textup{ where } (x,y)\rightarrow x-y , 
\]
which contracts the diagonal $\Delta\subseteq C\times C$. Let $F_1$ and $F_2$ be the corresponding fibers on $C\times C$. 

Let $\pi:\overline{J}_C\rightarrow J_C$ be the blow-up of the origin and $E\simeq \PP^{g-1}$ the exceptional divisor. Restricting it to $S_C$ and its proper transform $\overline{S}_C$ through $\pi$, it yields the commutative diagram
\[
\begin{tikzcd}
C\times C \arrow{r}{\overline{q}}  \arrow[rr, bend left=25, "\partial"] & \overline{S}_C \arrow{r}[swap]{\pi|_{\overline{S}_C}}    & S_C\\
C\times C\arrow[swap]{u}{\textup{id}\times \sigma} \arrow{r}{q}& \textup{Sym}^2(C) \arrow{u}{\sim}& 
\end{tikzcd}
\]
where $q$ is the natural quotient map, . The vertical maps are isomorphisms. 

By \cite{G84} or \cite[Proposition~11.1.4]{BL04}, the differential of $\phi_{K_C}$ provides the identification
\[
\overline{S}_C\cap E=Q\ \subseteq \ E\simeq \ \PP^{g-1}\ .
\]
Numerically, $\partial^*(\Theta_C|_{S_C})=(g-1)F_1+(g-1)F_2+\Delta$ and the graph of $\sigma$ is a curve 
\[\Gamma\ \in \ |2F_1+2F_2-\Delta|.
\]
Moreover, $\overline{q}(\Delta)=Q$ and $\overline{q}$ is ramified over the proper transform $\overline{F}_C$ of $F_C=\partial(\Gamma)$.

The final property we need is the "minimality" of $F_C$ in a numerical sense as a curve on $C\times C$.
\begin{lemma}\label{lem:hyperelliptic}
	For any irreducible curve $F\neq F_C\subseteq J_C$ the following inequality holds:
	\begin{equation}\label{eq:in}
	\frac{(\Theta_C\cdot F)}{\mult_0(F)}\ \geq \ 2 \ .
	\end{equation}
In particular, $\epsilon(\Theta_C)=\frac{(\Theta_C\cdot F_C)}{\mult_0(F_C)} = \frac{2g}{g+1}$, where $(\Theta_C\cdot F_C)=4g$ and $\mult_0(F_C)=2g+2$. 
\end{lemma}
\begin{proof}
The first step is to show that any curve $F$ that does not satisfy $(\ref{eq:in})$ has to be contained in the surface $S_C$. This follows directly from Welters' theorem \cite{W86}. It states that 
\[
\{x\in \textup{Supp}(D) \ | \ D\in |2\Theta_C|, \mult_0(D)\geq 4\} \ = \ S_C\cup\{\textup{isolated points}\}
\]
For any $D\in |2\Theta_C|$ with $\mult_0(D)\geq 4$ and any curve $C\nsubseteq \textup{Supp}(D)$, Bezout's theorem yields 
	\[
	2(\Theta_C\cdot C) \ = \ (D\cdot C)\ \geq \ \mult_0(D)\cdot \mult_0(C)\ \geq \ 4\mult_0(C) \ .
	\] 
As $F$ does not satisfy $(\ref{eq:in})$, then Bezout's and Welters' theorem force $F\subseteq S_C$. 
	
The second step is to transfer the data on $C\times C$ making use of the difference map. 
Based on the proof of Proposition~\ref{prop:etale}, even if the difference map is not \'etale, one can easily show that
	\[
	\Big(\big[(g-1)(F_1+F_2)+\Delta\big]\cdot \partial^*(F)\Big)\ = \ 2(\Theta_C\cdot F) \textup{ and } \big(\partial^*(F)\cdot \Delta\big)\ = \  2\mult_0(F) \ .
	\] 
As $\Gamma\nsubseteq \textup{Supp}(\partial^*(F))$, then on $C\times C$ these formulae yield the following inequalities 
	\[
	\Big(\big[(g-1)(F_1+F_2)+\Delta\big]\cdot \partial^*(F)\Big)\ = \ \Big(\big[\Gamma+(g-3)(F_1+F_2)\big]\cdot \partial^*(F)\Big)+2\Big(\Delta\cdot \partial^*(F)\Big)\ \geq \ 2\big(\partial^*(F)\cdot\Delta\big) \ .
	\]
This contradicts the initial assumption and finishes the proof.
\end{proof}

\subsection{Seshadri constant for étale cover over hyperelliptic Jacobians}
Making use of the infinitesimal geometry of hyperelliptic Jacobians described earlier, we can now study the behaviour of the Seshari constant of polarized abelian varieties that are ètale covers over them.
\begin{theorem}\label{thm:nfoldhype}
	Let $(A,L)$ be a polarized abelian variety of dimension $g\geq 3$ that admits a degree $k\geq 2$ étale map
	\[
	f:(A,L) \ \rightarrow \ (J_C,\Theta_C),
	\]
	such that $L=f^*(\Theta_C)$ and $(J_C,\Theta_C)$ is the Jacobian of a smooth hyperelliptic curve $C$ of genus $g$. Then $\epsilon(L) \geq  2$.
\end{theorem}

\begin{proof}
	We assume the statement doesn't hold, so the goal is to find a contradiction. By the definition of the Seshadri constant and \cite{N96}, this yields the existence of a curve $F\subseteq A$ with
\[
	1\ \leq \ \epsilon(L)\ \leq \ \frac{(L\cdot F)}{\mult_0(F)} \ < \ 2 \ .
\]
Applying Proposition~\ref{prop:etale} to the map $f:(A,L)\rightarrow (J_C,\Theta_C)$ yields the following inequalities:
\[
	2  > \frac{(L\cdot F)}{\mult_0(F)}  = \frac{\textup{deg}\big(F\rightarrow f(F)\big)\cdot(\Theta_C\cdot f(F))}{\textup{deg}\big(F\rightarrow f(F)\big)\cdot \mult_{0}(f(F))-\sum_{i=2}^{i=d}\mult_{x_i}(F)}\geq \frac{(\Theta_C\cdot f(F))}{\mult_0(f(F))}  ,
\]
	where $f^{-1}(0)=\{x_1=0,x_2,\ldots ,x_k\}$.
	
Taking into account Lemma~\ref{lem:hyperelliptic}, these inequalities imply $f(F)=F_C$ and $(\Theta_C\cdot f(F))=4g$. Moreover, $f$ is a local analytical isomorphism around the origin, so $\mult_{0}(F)\leq \mult_0(F_C)=2g+2$, and consequently $\textup{deg}(F\rightarrow f(F))=1$. Finally, the same inequalities and simple arithmetics show that we need to consider only two cases for our numerical data:
	\[
	\big((L\cdot F),\textup{mult}_0(F)\big)\ \in \ \{(4g,2g+2), (4g,2g+1)\} \textup{, thus } \epsilon(L)\in \{\frac{4g}{2g+2},\frac{4g}{2g+1}\} \ .
	\]
In the following let $m:=\mult_0(F)\in\{ 2g+1,2g+2\}$. In both cases the inequalities above imply that $F$ passes through $x_1=0$ with multiplicity $m$, but only for $m=2g+1$ it passes through another point in $f^{-1}(0)$ and is smooth there. As $f:A\rightarrow J_C$ is a quotient by a finite group, thus $f^*(F_C)$ cosists of exactly $k$ irreducible components, each a translation of $F$ that satisfy the same properties.

Consider now the fiber product
\[
\begin{tikzcd}
A' \arrow{r}{\pi'} \arrow[swap]{d}{f'}  & A \arrow{d}{f} \\
\overline{J}_C \arrow{r}{\pi}& J_C& 
\end{tikzcd}
\]
where $\pi$ is blow-up of the origin. Note $\pi'$ is the blow-up of the points $f^{-1}(0)=\{x_1=0,x_2,\ldots ,x_k\}$.  So, we have $k$ exceptional divisors $E_1,\ldots ,E_k$ on $A'$, with $E_1$ being the exceptional divisor $E_A\subseteq \overline{A}$.  Denote by $Q_i\subseteq E_i$ the pre-image of the quartic $Q\subseteq E\subseteq \overline{J}_C$ for each $i=1,\ldots ,k$.

Consider now the smooth surface $S'=(f')^*(\overline{S}_C)$. Let $(f')^*(\overline{F}_C)=F_1'\cup\ldots \cup F_k'$, where each $F_i'$ is the proper transform of a component of $f^*(F_C)$ through $\pi'$, including $F$. As the  restriction map $f'|_{S'}:S'\rightarrow \overline{S}_C$ remains étale of degree $k$, the smoothness of $\overline{F}_C$ implies that the components $F'_1,\ldots ,F'_k$ are disjoint between them and each is smooth. Thus, $(F_i'\cdot F_j')=0$ for any $i\neq j=1,\ldots ,k$. 

As $f$ is a quotient by a finite group, each component $F'_i$ is a translation of any other. Thus, $(F_i'^2)$ on $S'$ does not depend on the choice of the component. In particular, this implies
\[
(F_i'^{2})\ = \ \frac{1}{k}\cdot \big(f'^*(\overline{F}_C)^2\big) \ = \ (\overline{F}_C^2) \ = \  \frac{1}{2}\big(\partial^*(F_C)^2\big) \ = \  \frac{1}{2}\big((2\Gamma )^2\big) \ = \ 4-4g \ .
\]
The same trick implies $(Q_i\cdot Q_j)=0$  for any $i\neq j=1,\ldots ,k$ and all $(Q_i^2)$ are equal. In particular, 
\[
(Q_1^2)\ = \ (Q^2) \ = \ \frac{1}{2}(\Delta^2)=1-g.
\]
Denoting $L_S\deq(\pi')^*(L)|_{S'}$, then $L_S^2=k\cdot g(g-1)$ and $(L_S\cdot Q_1)=0$. Moreover, as $F_1$ is isomorphic with $\overline{F}$ and $Q_1$ with $Q$, then $(L_S\cdot F'_1)=4g$ and $(F'_1\cdot Q_1)=m$.

Finally, we get our contradiction by applying Lemma \ref{lem:boundary} to the Newton-Okounkov polygon $\Delta_{(Q_1,P)}(L_S)$, where $P\in Q_1$ is a generic point. To do so, note first that $L_S-tQ_1$ is ample on $(0,\frac{4g}{m})$, being the restriction of $(\pi')^*(L)-tE_1$ to $S'$, which is an ample class on $A'$ in this range. Second, by Lemma \ref{lem:hyperelliptic} and the above discussion, $F_1'$ is the only curve appearing in the Zariski decomposition of the class $L_S-tQ_1$ for any $t\in (\frac{4g}{m},2)$. So, applying Zariski's algorithm from \cite[Theorem 14.14]{B01} and \cite[Lemma 2.5]{Loz18}, we get the lower bound
\[
\mult_{F_1'}(||L_S-tQ_1||)\ \geq \ \frac{mt-4g}{4g-4} \ = \ \frac{m}{4g-4}\cdot \Big(t-\frac{4g}{m}\Big), \ \ \forall t\in \Big(\frac{4g}{m},2\Big) \ .
\]
Now, by Lemma \ref{lem:boundary}, the set $\Delta_{(Q_1,P)}(L_S)$ is contained in the triangle with vertices: $(0,0)$, $(\frac{4g}{m},\frac{4g(g-1)}{m})$ and $(\frac{4gm}{m^2-4(g-1)^2},0)$. The area of the latter is $\frac{8g^2(g-1)}{m^2-4(g-1)^2}$ and is less than that of $\Delta_{(Q_1,P)}(L_S)$, which is $\frac{(L_S)^2}{2}=\frac{kg(g-1)}{2}$  by \cite{LM09}, whenever $k\geq 2$ and $g\geq 3$. Thus we get our final contradiction.
\end{proof}

\section{\'Etale covers over non-hyperelliptic Jacobians in dimension three}

In this section we study the behaviour of the Seshadri constant for an abelian polarized  threefold $(A,L)$ that is an étale cover over the Jacobian of a non-hyperelliptic curve of genus $g=3$.

We start by reminding the reader of the basic positivity results on such Jacobians and then continue with the proof of the main statement.

\subsection{Three-dimensional non-hyperelliptic Jacobians}
We start with a short review of the classical infinitesimal picture for the Jacobian for a non-hyperelliptic curve $C$ of genus three. 

By \cite[Section 11.1-11.2]{BL04} there is a canonical isomorphism $J_C\simeq \Pic^0(C)$, so that the Abel-Jacobi map provides us with an embedding $C\hookrightarrow J_C$, defined as $P\in C\mapsto \sO_C(P-P_0)\in \textup{Pic}^0(C)$ for some fixed $P_0\in C$. Moreover, we have the identification $\Theta_C\simeq C+C\simeq \textup{Sym}^2(C)$. 

The difference map $\partial :C\times C\rightarrow J_C$ defines a surface $S_C\deq C-C\in |2\Theta_C|$ with $\textup{mult}_0(S_C)=4$ and smooth elsewhere (see \cite{GG86}). Let $\pi:\overline{J}_C\rightarrow J_C$ be the blow-up at the origin with exceptional divisor $E\simeq \PP^2$ and let $\overline{S}_C$ be the proper transform of $S_C$. Then the Gauss map of the Abel-Jacobi map (or its differential) is the canonical map of $C$, i.e. the embedding
\[
 C \ \longrightarrow \ \overline{S}_C\cap E=Q\ \subseteq \ E\simeq \ \PP^2 \ ,
\]
as a planar quartic, see \cite{G84} or \cite[Proposition~11.1.4]{BL04}. Moreover, note that $\overline{S}_C\simeq C\times C$, $\partial= \pi|_{\overline{S}_C}:\overline{S}_C\rightarrow S_C$ is the blow down of the diagonal $\Delta\subseteq C\times C$, and $\partial^*(\Theta_C|_{S_C})=2F_1+2F_2+\Delta$. 

By \cite{BS01} the Seshadri contant $\epsilon(\Theta_C)$ is computed by a symmetric curve $F_C$, whose proper transform $\overline{F}_C\in |16F_1+16F_2-6\Delta|$. Geometrically, each tangent line of $Q$ intersects $Q$ in two other points, defining two symmetrically equivalent points on $F_C$. This provides the description
\[
\epsilon(\Theta_C;0) \ = \ \frac{(\Theta_C\cdot F_C)}{\mult_0(F_C)} \ = \ \frac{96}{56}\ = \  \frac{12}{7} \ .
\]
Unfortunately, there is no minimality property as in Lemma \ref{lem:hyperelliptic} for non-hyperelliptic curves. But for our purposes it will suffice to provide some helpful arithmetic properties for intersection numbers.
\begin{lemma}\label{lem:nonhyperelliptic}
If $F$ is a symmetric curve on $C\times C$, then $(F\cdot\Delta)$ is divisible by 2.
\end{lemma}
\begin{proof}
Let $q:C\times C\rightarrow \textup{Sym}^2(C)$ be the natural quotient map. If $F=\Delta$, then $F\cdot \Delta=-4$ and is divisible by 2. Suppose $F\neq\Delta$ and let $F'=q(F)$ and $\overline{F}_C'=q(\overline{F}_C)$. Since both $F$ and $\overline{F}_C$ are symmetric, then $F=q^*F'$ and $q^*\overline{F}_C'=\overline{F}_C$. Thus, projection formula yields $F\cdot\overline{F}_C=2\cdot F'\cdot\overline{F}_C'$.

Let's show that the class $\overline{F}_C'$ is divisible by $2$ in $\textup{Pic}(\textup{Sym}^2(C))$. Indeed, consider the involution $\textup{Sym}^2(C)\to \textup{Sym}^2(C)$, sending $x+y\in \textup{Sym}^2(C)$ to $\overline{xy}\cap C-(x+y)$, where $\overline{xy}$ is the line through $x$ and $y$ in $\mathbb P^2$.
This involution sends $\Delta'=q(\Delta)$ to $\overline{F}_C'$. 
However, $\frac{\Delta'}{2}$ is a Cartier divisor on $\textup{Sym}^2(C)$ (not effective, but a $\mathbb Z$-divisor nonetheless) and its pullback is $\Delta$, as it is the branching locus of $q$.

In particular, $F'\cdot\overline{F}_C'$ is even and $F\cdot\overline{F}_C$ is divisible by 4. By the class description of $\overline{F}_C$, this easily implies that $(F\cdot \Delta)$ is divisible by $2$.
\end{proof}

\subsection{Seshadri constant for ètale covers over non-hyperelliptic Jacobians in dimension three}
Making use of the infinitesimal geometry of non-hyperelliptic Jacobians in dimension three, we study the behaviour of the Seshari constant of polarized abelian three-folds that are ètale covers over them.
\begin{theorem}\label{thm:threefoldnohype}
	Let $(A,L)$ be a polarized abelian threefold that admits a degree $k\geq 2$ étale map
	\[
	f:(A,L) \ \rightarrow \ (J_C,\Theta_C),
	\]
with $L=f^*(\Theta_C)$ and $(J_C,\Theta_C)$ the Jacobian of a smooth non-hyperelliptic curve $C$ of genus $g=3$. Then $\epsilon(L) \geq  2$.
\end{theorem}

\begin{proof}
We assume the statement doesn't hold, so the goal is to find a contradiction. By the definition of the Seshadri constant and \cite{N96}, this yields the existence of a curve $F\subseteq A$ with
	\begin{equation}\label{eq:cond1}
	1\ \leq\ \epsilon(L)\ \leq \ \frac{(L\cdot F)}{\mult_0(F)} \ < \ 2 \ .
	\end{equation}
By \cite[Exercise 2.6.10]{BL04}, both $L$ and $(-1_A)^*(L)$ are numerically equivalent. Thus we can assume $F$ is symmetric. And in this case it is either irreducible or consists of two irreducible components, and each has the same numerical quotient from $(\ref{eq:cond1})$ as $F$. 

The existence of the étale cover to the non-hyperelliptic Jacobian, provides us with a surface $S_A\deq f^*(S_C)\in |L^{\otimes 2}|$, which is smooth everywhere besides the points in $f^{-1}(0)=\{0,x_2,\ldots ,x_k\}$ where it has multiplicity $4$. In particular, B\'ezout's theorem and $(\ref{eq:cond1})$ force $F\subseteq S_A$. 

Let $\overline{\Gamma}$ be the proper transform of $\Gamma\deq f(F)$ by the blow-up $\pi:\overline{J}_C\rightarrow J_C$. It is simmetric and is either irreducible or consists of two irreducible components. Thus, $(\ref{eq:cond1})$ and Proposition \ref{prop:etale} yield
	\begin{equation}\label{eq:inenonhype}
	\begin{split}
	2  > \frac{(L\cdot F)}{\mult_0(F)} & = \frac{\textup{deg}\big(F\rightarrow f(F)\big)\cdot(\Theta_C\cdot f(F))}{\textup{deg}\big(F\rightarrow f(F)\big)\cdot \mult_{0}(f(F))-\sum_{i=2}^{i=d}\mult_{x_i}(F)}\\ &\geq \frac{(\Theta_C\cdot \Gamma)}{\mult_0(\Gamma)} = \frac{(2F_1+2F_2+\Delta)\cdot \overline{\Gamma}}{\Delta\cdot \overline{\Gamma}} \geq \frac{12}{7} .
	\end{split}
	\end{equation}
With these inequalities in hand, let's find some arithmetic conditions on the numbers involved. First, as $f$ is locally an analytic isomophism, then  $\textup{mult}_0(F)\leq \textup{mult}_0(\Gamma)$. Combining this with $(\ref{eq:inenonhype})$ force $\textup{deg}\big(F\rightarrow \Gamma\big)=1$. Second, as $\overline{\Gamma}$ is symmetric (one or two components) sitting on $\overline{S}_C=C\times C$, Lemma \ref{lem:nonhyperelliptic} and $(\ref{eq:inenonhype})$ imply that there exists strictly positive integers $r, s\in \ZZ$ such that
\[
\mult_0(\Gamma) =  (\Delta\cdot \overline{\Gamma}) = 4r+2s \geq 10, \ (F_1\cdot \overline{\Gamma})=(F_2\cdot \overline{\Gamma}) = r \geq 2 \ , \textup{ and }r \geq \frac{5}{4}s >0 \ .
\]
This description imposes conditions on the pair $((L\cdot F),q\deq\mult_0(F))$. As $(L\cdot F)=(\Theta_C\cdot f(F))$, then $(L\cdot F)$ is even by Lemma \ref{lem:nonhyperelliptic}. Furthermore, $(\ref{eq:inenonhype})$ implies $(L\cdot F)\geq 18$ and $q\geq 10$. A more careful analysis allows us to discard the cases when $q=11,12,13$. Indeed, as $(L\cdot F)$ is even, the description allows only $(20,11),(22,12)$ and $(24,13)$ to be pairs for $((L\cdot F),q)$. But in all three cases $(\ref{eq:inenonhype})$ forces $q=(\Delta \cdot \overline{\Gamma})$. So, $q$ is even and this discards the first and third example. Similarly for the second example, as the description above would imply $4r=10$.

Similar ideas for the cases $q=10,14,15$, lead us to the following description of possible pairs:
\begin{equation}\label{eq:possible}
((L\cdot F),\mult_0(F)) \ \in \ \{(18,10),(26,14), (28,15)\}\bigcup \{((L\cdot F), q) \ | \ q\geq 16\},
\end{equation}
dividing our possibilities into three special cases and the general one.

For most of the cases we will get our contradiction by studying a non-standard infinitesimal Newton-Okounkov body of $L$ at the origin. The remaining cases will be dealt by making use of similar ideas, as the ones developed in the proof of Theorem \ref{thm:nfoldhype}.

In order to start, let $\pi_A:\overline{A}\rightarrow A$ be the blow-up of the origin with $E_A\simeq \PP^2$ the exceptional divisor. Let $\overline{S}_A$ be the proper transform of  $S_A$ through $\pi_A$ and $\overline{F}$ of $F$. The intersection $Q=\overline{S}_A\cap E_A\subseteq \PP^2$ is a smooth quartic. Consider the flag $Y_{\bullet}: \overline{A}\supseteq E_A\supseteq Q\supseteq \{P\}$ for some generic point $P\in Q$. The goal is to study the shape of the non-standard infinitesimal Newton-Okounkov body
\[
\Delta(L) \ \deq \  \Delta_{Y_{\bullet}}(\pi^*_A(L)) \ \subseteq  \ [0,\mu(L)]\times \RR^2 \ .
\]
Getting inspired by \cite{KL17}, for any $\mu \geq t\geq 0$, we have the inclusion
\begin{equation}\label{eq:slice}
\Delta_t(L) \ \deq \ \Delta(L)\ \cap \{t\}\times \RR^2 \ \subseteq \ \{(t,x,y)\ | \ x,y\geq 0, 16x+y\leq 4t\} \ .
\end{equation}
This follows by making use simultaneously of the valuative construction of Newton-Okounkov bodies, Bézout's theorem and the discussion in \cite[Remark 1.8]{KL18}.

By \cite{LM09} the area of each slice $\Delta_t(L)$ is exactly $\frac{1}{2}\cdot \textup{vol}_{\overline{A}|E_A}(B_t)$, where the latter is called the restricted volume and encodes asymptotically the dimension of global sections of powers of the class $B_t=\pi_A^*(L)-tE_A$ that can be restricted non-trivially to $E_A$ (see \cite{ELMNP09} and \cite{LM09} on the basic theory). As a consequence of \cite[Corollary C]{LM09}, we get the following formula
\begin{equation}\label{eq:main}
k \ \deq \ \textup{deg}(f)\  = \ \frac{L^3}{6}\ = \ \int_{0}^{\infty} \textup{Area}_{\RR^{2}}\big(\Delta_t\big) dt \ = \  \int_{0}^{\infty} \frac{1}{2}\cdot \textup{vol}_{\overline{A}|E}(B_t) dt \ .
\end{equation}
The contradiction will follow from $(\ref{eq:main})$, if we find good upper bounds on $\Delta_t(L)$ for any $t\in [0,\mu]$.

But before doing so, let's apply first these ideas and get the following useful fact:
\begin{equation}\label{eq:surfacebaselocus}
\overline{S}_A\ \nsubseteq \ \Bplus(B_t) \textup{ for any } t\in (\epsilon, 2] \ .
\end{equation}
Indeed if $\overline{S}_A\subseteq \Bplus(B_{t_0})$ for some $t_0\in(\epsilon,2]$, \cite[Lemma 1.3]{Loz18} implies the same for any $t\geq t_0$. As $B_2\equiv_{\textup{num}}\overline{S}_A$, this implies $\mu= 2$. Consequently, $(\ref{eq:slice})$ forces $\Delta(L)$ to be contained in the tetrahedron generated by the vertices $(0,0,0),(2,0,0),(2,\frac{1}{2},0)$ and $(2,0,8)$. We get our contradiction, as the the latter has volume equal $\frac{8}{6}$ and the former $\frac{L^3}{6}\geq 2$.

Turning our attention to our convex sets, let's note first that on $[0,\epsilon)$ the class $B_t$ is ample, so  \cite[Corollary~2.17]{ELMNP09} and \cite{LM09} yield
	\begin{equation}\label{eq:volume1}
	2\cdot \textup{vol}_{\RR^2}(\Delta_t(L)) \ = \ \textup{vol}_{\overline{A}|E_A}(B_t)\ = \ (B_t^2\cdot E_A) \ = \ t^2 \ \textup{ for any } t\in [0,\epsilon]. 	
	\end{equation}
In particular, both sets in $(\ref{eq:slice})$ have equal area and the inclusion there is actually an equality.

On $[\epsilon,\mu]$ the locus $\Bplus(B_t)$ contains $\overline{F}$, so \cite[Proposition 3.3]{Loz18} yields the following inequality:
	\begin{equation}\label{eq:nak1}
	\mult_{\overline{F}}(||B_t||) \ \geq \ t-\epsilon\ , \textup{ for any } t\in [\epsilon,\mu] \ ,
	\end{equation}
	where $\overline{F}$ is either irreducible or the union of two components.
	
Under these circumstances, our goal is to show the following claim:
\begin{claim}
We have the following inclusion:
\begin{equation}\label{eq:claim}
\Delta_t(L)\ \subseteq \ \Big\{(t,x,y)\ | \ x, y\geq 0, \ y+16x-4t\leq \textup{min}\big\{0, q\big(x-(t-\epsilon)\big)\big\} \Big\} \ ,
\end{equation}
for any $t\in [\epsilon, \mu]$.
\end{claim}
\begin{proof}
The reason this works is the intersection $Q=\overline{S}_A\cap E_A$ being transversal. Moreover, $\overline{F}$ is a Cartier divisor in $\overline{S}_A$, so in some sense the entire picture is similar to complete intersections. 

So, there is a cover of open Zariski-neighborhoods $U\subseteq \overline{A}$, such that on each $U$, we can assume without loss of generality that there is a local system of coordinates $\{u_1,u_2,u_3\}$, such that 
\[
\overline{S}_A|_U\ =\ \{u_2=0\} \ \textup{ and } E_A|_U\ =\ \{u_3=0\} \ .
\]
In this setup then $\overline{F}|_U=\{u_2=f(u_1,u_3)=0\}$ for some polynomial $f\in \CC[u_1,u_3]$ that might be reducible with at most two distinct irreducible components. 

Consider an effective Cartier divisor $\overline{D}\equiv_{\textup{num}} mB_t$ for some divisible enough $m\gg 0$, such that $E_A\nsubseteq \textup{Supp}(\overline{D})$. By $(\ref{eq:nak1})$ we know
\[
\mult_{\overline{F}}(\overline{D}) \ \geq \ m(t-\epsilon) \ . 
\]
Using an inductive process, the definition of the multiplicity through differential operators in variables $\{u_1,u_2,u_3\}$, and the restriction to $\overline{S}_A|_U$, we can deduce the following local description
\[
D|_U \ = \ \sum_{i=0}^{i=m(t-\epsilon-s)}g_i\cdot u_2^{m(t-\epsilon)-i}\big(f(u_1,u_3)\big)^i \ ,
\]
for some $s\leq t-\epsilon$ with each $g_i\in \CC[u_1,u_2,u_3]$ and $g_{m(t-\epsilon-s)}\neq 0$. Note here that for this to work we needed the curve $\overline{F}|_U$ to be a complete intersection as otherwise the connection between symbolic powers of ideals and powers of the actual ideal is much more complicated. In particular,
\[
\overline{D}|_{E_A\cap U} \ = \ \sum_{i=0}^{i=m(t-\epsilon-s)}g_i(u_1,u_2,0)\cdot u_2^{m(t-\epsilon)-i}\big(f(u_1,0)\big)^i \ .
\]
Taking into account this for all open sets $U$ of the cover, and taking $P\in Q$ general, thus sitting in all $U$, it remains to show that the valuation vector $\nu_{Y_{\bullet}}(\overline{D})=(t,\nu_2,\nu_3)$ satisfies the inequalities
\[
\nu_2 \ \geq \ ms \ \textup{ and } \nu_3 \ \leq \  m(4t-16\nu_2) \ - \ q\cdot \big(m\cdot (t-\epsilon)-\nu_2\big) \ ,
\]
that easily imply the statement. So, first note that the restriction of $D'\deq \overline{D}|_{E_A}\ -\nu_2\textup{div}(Q)$ to $Q$ has degree exactly $m(4t-16\nu_2)$ on $Q$. By the final local description, the divisor $D'|_Q$ vanishes at each point $R\in \overline{F}\cap Q$ with multiplicity at least $m(t-\epsilon)(\overline{F}\cdot Q)_R-\nu_2$, where $(\overline{F}\cdot Q)_R$ is the local intersection multiplicity at point $R$ of $\overline{F}$ with $Q$, as divisors on $\overline{S}_A$. As $P\notin \overline{F}\cap Q$, these ideas imply the desired inequalities. It's worth pointing out that  these ideas work out as it did earlier because our main players are local complete intersection in each $U$ even when $F$ is reducible.
\end{proof}
Moving forward, we divide the proof in three distinct cases based on the data from $(\ref{eq:possible})$.

$\textit{\underline{Case 1}:}$ $q=mult_0(F)\geq 16$.

By $(\ref{eq:claim})$, $\Delta_t(L)$ sits below the line $y=4t-16x-q(t-\epsilon-x)$ of positive slope as $q\geq 16$. It intersects the line $y=4t-16x$ in $B_t$, the $y$-axis in $C_t$ and the $x$-axis in $D_t$, described as follows:
\[
B_t=(t, t-\epsilon, 16\epsilon-12t), \ C_t=(t, 0, (4-q)t+q\epsilon), \textup{ and } D_t=(t, \frac{(q-4)t-q\epsilon}{q-16}, 0)  \ .
\]
Finally, let's denote by $O_t=(t,0,0)$ and $A_t=(t,\frac{t}{4},0)$. 

With this in hand, note that  both $B_t$ and $C_t$ are below the $x$-axis for $t\geq \frac{4}{3}\epsilon$, so $(\ref{eq:claim})$ yields $\mu(L)\leq \frac{4}{3}\epsilon$. Also $D_t$ has positive $x$-coordinate for any $t\geq \frac{q}{q-4}\epsilon$. In particular, these ideas yield the following inclusions
\[
\Delta_t(L) \ \subseteq \ \begin{cases} 
                       O_tA_tB_tC_t, \ \epsilon\leq t\leq \frac{q}{q-4}\epsilon\\
                       D_tA_tB_t, \ \frac{q}{q-4}\epsilon\leq t \leq \frac{4}{3}\epsilon
                       \end{cases} \ .
\]
Computing the areas and taking into account $(\ref{eq:main})$ and \cite{LM09}, we obtain the following inequality
\[
	2\leq k=\frac{L^3}{6} \ \leq \ \int_{0}^{\epsilon}\frac{t^2}{2} \ dt\ + \  \int_{\epsilon}^{\frac{q\epsilon}{q-4}}\frac{t^2-q(t-\epsilon)^2}{2} \ dt\ + \ \int_{\frac{q\epsilon}{q-4}}^{\frac{4}{3}\epsilon}\frac{q(4\epsilon-3t)^2}{2(q-16)} \ dt  \ .
\]
Note here that when $q=16$ the third integral does not appear. So, computing these integrals give us a much simpler inequality that works also for $q=16$.
\[
2 \ \leq \ \frac{\epsilon^3}{6} \ + \ \frac{2(3q^2-28q+16)}{3(q-4)^3}\epsilon^3\ + \ \frac{q(q-16)^2}{18(q-4)^3}\epsilon^3 \ = \ \frac{2q}{9(q-4)}\epsilon^3 \ .
\]
As  $(L\cdot F)=(\Theta_C\cdot f(F))$, then $\epsilon(L)=\frac{2q-2l}{q}$ for some $l\geq 1$ by Lemma \ref{lem:nonhyperelliptic}. Setting it in our last inequality leads to one that doesn't hold for $q\geq 16$. Thus, we get our contradiction in this case.

$\textit{\underline{Case 2}:}$ $((L\cdot F),q)\in \{(26,14),(28,15)\}$.

By $(\ref{eq:claim})$, $\Delta_t(L)$ sits below the line $y=4t-16x-q(t-\epsilon-x)$ of negative slope as $q< 16$. Let the points $O_t,A_t,B_t,C_t, D_t$ be as previously. In this case both $B_t$ and $C_t$ are below the $x$-axis whenever $t\geq \frac{q}{q-4}\epsilon$ and thus $\mu(L)\leq \frac{q}{q-4}\epsilon$. Furthermore, $D_t$ is to the left of $A_t$ while $t\geq \frac{4}{3}\epsilon$. Applying $(\ref{eq:claim})$ and convex geometry leads in this case to the following inclusions
\[
\Delta_t \ \subseteq \ \begin{cases} 
                       O_tA_tB_tC_t, \ \epsilon\leq t\leq \frac{4}{3}\epsilon \\
                       O_tD_tC_t, \ \frac{4}{3}\epsilon\leq t \leq \frac{q}{q-4}\epsilon
                       \end{cases} \ .
\]
By the same exact ideas as in the previous case, we are lead to the following inequality
\[
	2\leq k=\frac{L^3}{6} \ \leq \ \int_{0}^{\epsilon}\frac{t^2}{2} \ dt\ + \  \int_{\epsilon}^{\frac{4}{3}\epsilon}\frac{t^2-q(t-\epsilon)^2}{2} \ dt\ + \ \int_{\frac{4}{3}\epsilon}^{\frac{q}{q-4}\epsilon}\frac{(q\epsilon-(q-4)t)^2}{2(16-q)} \ dt \ .
\]
Computing the integrals above this yeilds then the inequality
\[
2 \ \leq \ \frac{\epsilon^3}{6} \ + \ \Big(\frac{(64-q)}{162}\epsilon^3-\frac{\epsilon^3}{6}\Big)\ + \ \frac{(16-q)^2}{162(q-4)}\epsilon^3 \ = \ \frac{2q}{9(q-4)}\epsilon^3
\]
that fails easily in both of our examples from this case.

$\textit{\underline{Case 3}:}$ $((L\cdot F),q)=(18,10)$.

First note that the steps taken in Case $2$ lead us again to a contradiction whenever $k=L^3/6\geq 3$. So, it remains to tackle the case $L^3=12$, i.e. $f:(A,L)\rightarrow (J_C,\Theta_C)$ has degree two. Moreover, $F$ is symmetric, so it is either irreducible with $(L\cdot F)=18$ and $\mult_{0}(F)=10$, or consists of two components $G_1$ and $G_2$ such that $(L\cdot G_i)=9$ and $\mult_0(G_i)=5$ for $i=1,2$. 

First, let's get a contradiction for the latter case. Let $\overline{\Gamma}_i$ be the proper transform through the blow-up $\pi:\overline{J}_C\rightarrow J_C$ of $f(G_i)$ for any $i=1,2$. By $(\ref{eq:inenonhype})$ note $f^*(f(G_i))$ consists of two irreducible components that don't intersect, implying $(2F_1+2F_2+\Delta\cdot \overline{\Gamma}_i)=9$ and $(\Delta\cdot \overline{\Gamma}_i)=5$. In particular, we have $(F_1\cdot \overline{\Gamma}_i)=(F_2\cdot \overline{\Gamma}_i)=1$, So, each $\overline{\Gamma}_i$ is the graph of an automorphism of $C$. In particular, a local description of the graph implies that each $\overline{\Gamma}_i$ intersects transversally $\Delta$ at exactly $5$ distinct points. As $\overline{\Gamma}_1$ is the symmetric image of $\overline{\Gamma}_2$, then this forces $(\overline{\Gamma}_1\cdot \overline{\Gamma}_2)\geq 5$. Consequently, the class $\overline{\Gamma}_1+\overline{\Gamma}_2$ must be big and nef on the surface $\overline{S}_C$. As $f^*(f(G_i))$ consists of two disjoint components, then $\overline{F}=\overline{G}_1+\overline{G}_2$ must be itself big and nef on the surface $\overline{S}_A$. As $(B_t\cdot \overline{G}_1+\overline{G}_2)<0$ for any $t>1.8$, Lemma \ref{lem:curvesurface} implies $\overline{S}_A\subseteq \Bplus(B_t)$ for any $t>1.8$, contradicting $(\ref{eq:surfacebaselocus})$.

It remains to deal with the case when $F$ is irreducible and symetric with $(L\cdot F)=18$ and $\mult_{0}(F)=10$. We can also assume that $F$ is the only curve through $0\in A$, satisfying the inequality
\begin{equation}\label{eq:condition}
\frac{(L\cdot F)}{\textup{mult}_0(F)}\ < \ 2 ,
\end{equation}
as otherwise we can reduce ourselves to the cases treated above. Moreover, $(\ref{eq:inenonhype})$ implies  $f^*(\Gamma)$ is the disjoint union of $F$ and another curve $F'$ with both isomorphic with $\Gamma\subseteq J_C$ through $f$.

The rest is similar to the proof of Theorem \ref{thm:nfoldhype}. So, consider the fiber product
\[
\begin{tikzcd}
A' \arrow{r}{\pi'} \arrow[swap]{d}{f'}  & A \arrow{d}{f} \\
\overline{J}_C \arrow{r}{\pi}& J_C& 
\end{tikzcd}
\]
where $\pi$ is blow-up of the origin of $J_C$. Note $\pi'$ is the blow-up of the two points in $f^{-1}(0)$.  So, we have two exceptional divisors $E_1,E_2$ on $A'$, where the first one is the exceptional divisor $E_A$ on $\overline{A}$. Let $Q_1\subseteq E_1$ and $Q_2\subseteq E_2$ be the pre-images of the quartic $Q\subseteq E_A$.

Consider now the smooth surface $S'=(f')^*(\overline{S}_C)$. Note $(f')^*(\overline{\Gamma})=\overline{F}\sqcup\overline{F}'$, where $\overline{F}$ and $\overline{F}'$ the proper transform through $\pi'$ of $F$ and respectively $F'$. Since $f'$ derives from an involution of $A$, sending $\overline{F}$ into $\overline{F}'$, then $(\overline{F}^2)=(\overline{F}'^2)=(\overline{\Gamma}^2)$. Consequently, adjunction formula yields
\[
2p_a(\overline{F})-2=2p_a(\overline{\Gamma})-2=(K_{C\times C}\cdot \overline{\Gamma})\ + \ (\overline{\Gamma}^2)=16+(\overline{\Gamma}^2),
\]
since $(F_1\cdot \overline{\Gamma})=(F_2\cdot \overline{\Gamma})=2$ on $\overline{S}_C=C\times C$. As one projection $C\times C\rightarrow C$ gives rise to a degree two map $\overline{\Gamma}\rightarrow C$, then Riemann-Hurwitz, as in \cite{C84}, implies then
\[
2p_a(\overline{\Gamma})-2=2(2g(C)-2)\ + \ R \ = \ 8\ +\ R \ ,
\]
for some integer $R\geq 0$, computed from the ramification divisor on the normalization of $\overline{\Gamma}$ and an effective divisor acquired from the singular locus. Consequently, we have $a:= (\overline{F}^2)=(\overline{\Gamma}^2)\geq -8$.

Moving forward note first that we can assume $a<0$, as otherwise $\overline{F}$ defines a nef class on $S'$, intersecting negatively $(\pi')^*(L)-tE_1$ for any $t>1.8$. So, by Lemma \ref{lem:curvesurface}, $S'\subseteq \Bplus((\pi')^*(L)-tE_1)$ in this range, contradicting $(\ref{eq:surfacebaselocus})$.  

It remains to deal with the cases when $a\deq-2,-4,-6,-8$. So, denote by $L_S\deq(\pi')^*(L)|_{S'}$. The same trick with the involution implies $(Q_1\cdot Q_2)=0$ and $(Q_1^2)= (Q_2^2)= (\Delta^2)=-4$. As $\overline{F}$ is isomorphic with $\overline{\Gamma}$ and $Q_1$ with $Q$, then $(L_S\cdot Q_1)=0$, $(L_S\cdot \overline{F})=18$ and $(\overline{F}\cdot Q_1)=10$.

Finally, we apply Lemma \ref{lem:boundary} to the Newton-Okounkov polygon $\Delta_{(Q_1,P)}(L_S)$, where  and $P\in Q_1$ some generic point. For $t\in(0,1.8)$ the class $L_S-tQ_1$ is ample, since it's the restriction to $S'$ of $(\pi')^*(L)-tE_1$, that is an ample class on $A'$ in this range. For $t\in(1.8,2)$ the curve $\overline{F}$ starts appearing in the Zariski decomposition of $L_S-tQ_1$. By unicity from $(\ref{eq:condition})$, it is the only negative curve appearing in the Zariski decomposition in this range. Applying Zariski's algorithm, as in \cite[Theorem 14.14]{B01}  with \cite[Lemma 2.5]{Loz18}, yields
\[
\textup{mult}_{\overline{F}}(||\overline{L}_S-tQ_1||) \ \geq \ \frac{(L_S-tQ_1\cdot \overline{F})}{(\overline{F}^2)} \ = \ \frac{18-10t}{a} \ = \ -\frac{10}{a}(t-1.8) \ \ \forall t\in (1.8,2) \ . 
\]
By Lemma \ref{lem:boundary} the set $\Delta_{(Q_1,P)}(L_S)$ is contained in the triangle with vertices $(0,0)$, $(1.8,7.2)$ and $(\frac{45}{a+25},0)$ and area equal to $\frac{162}{a+25}$. This is less than the area of $\Delta_{(Q_1,P)}(L_S)$, which by \cite{LM09} is equal $\frac{(L_S)^2}{2}=12$, for all $a=-2,-4,-6,-8$. Thus we get our final contradiction. 

\end{proof}

\section{Proof of Theorem \ref{thm:main} and Corollary \ref{cor:main}}
In this section we first study the behaviour of the Seshadri constant on abelian surfaces, by providing a very short proof, without the need of the ideas from previous sections. We finish the section by providing a proof of Theorem \ref{thm:main} and Corollary \ref{cor:main}. 

\subsection{Seshadri constants on abelian surfaces.}
In this subsection we study the Seshadri constant on polarized abelian surfaces. The main result is probably known to experts, i.e. see \cite{BS98}, but we include it here  to contrast it to the difficulty encountered for the higher-dimensional case.
\begin{proposition}\label{prop:surface}
	Let $(S,L)$ be an indecomposable polarized abelian surface that is not principle. Then $\epsilon(L)\geq 2$. 
\end{proposition}
\begin{remark}\label{rem:smooth}
	If $L=\Theta_S$ is an irreducible theta divisor then $\epsilon(L)=\frac{4}{3}$ by \cite{St98}. Moreover $\Theta_S$ is smooth, as otherwise the condition $\Theta_S^2=2$ forces $\epsilon(\Theta_S;0)\leq 1$. By applying then \cite{N96} this contradicts that $\Theta_S$ is irreducible. 
\end{remark}
\begin{remark}\label{rem:surface}
	Looking more carefully at the proof and making use of \cite{N96}, we deduce that for a $(S,L)$ a polarized abelian surface that is not principle, then either $\epsilon(L)=1$ and there exists an elliptic curve $F\subseteq S$ with $(L\cdot F)=1$ or $\epsilon(L)\geq 2$. 
\end{remark} 
\begin{proof}[Proof of Proposition~\ref{prop:surface}]
	Based on asymptotic Riemann-Roch and Remark~\ref{rem:smooth}, we can assume that $(L^2)=2k\geq 4$. Now, let $\epsilon(L)<2$,  and the goal is to get a contradiction.
	
	By Nakai-Moishezon criterion, used on the blow-up at the origin, the condition that $L^2\geq 4$ implies the existence of an irreducible curve $C\subseteq S$ with $q=\mult_0(C)$, $p=(L\cdot C)$, and
	\[
	1 \ < \ \epsilon(L;0)\ = \ \frac{p}{q}\ < \ 2 \ ,
	\]
	where the first inequality is due to \cite{N96}. Furthermore, \cite{KSS09}, yields $C^2\geq q^2-q+2$. Applying this statement together with Hodge index, we then get the following string of inequalities
	\[
	4 \ \leq \ 2k=(L^2) \ \leq \ \frac{(L\cdot C)^2}{C^2}\ \leq \ \epsilon(L;0)^2\cdot \frac{q^2}{q^2-q+2} \ = \ \frac{p^2}{q^2-q+2} \ . 
	\]
	As $p$ is a positive integer and $\frac{p}{q}<2$, then $p\leq 2q-1$. Plugging this upper bound into the expression on the right leads to an inequality that doesn't hold for any $q\geq 1$. This finishes the proof.
\end{proof}

\subsection{Proof of Theorem \ref{thm:main} and Corollary \ref{cor:main}}
The main goal of this subsection is to finish the proof of our main results from the introduction.

\begin{proof}[Proof of Theorem \ref{thm:main}]
When $g=2$, the statement is proved in	Proposition~\ref{prop:surface}. In the following we let $g\geq 3$. We will assume that the statement holds in any dimension at most $g-1$ and Conjecture \ref{conj:main} holds for any principle polarizartion of dimension $g$. The goal is to show that for any idencomposable polarized abelian variety $(A,L)$ of dimension $g$ that is not principle we have $\epsilon(L)\geq 2$. 

So fix such a pair $(A,L)$. Since it is not principle then by \cite[Proposition~4.1.2]{BL04} yields that for any ample line bundle $L$ on $A$, there is a degree $k\geq 2$  \'etale map of degree
\[
f: (A,L)\ \longrightarrow \ (A_L,\Theta_L) \ ,
\]
such that $f^*(\Theta_L)=L$ and $\Theta_L$ is a theta divisor on $A_L$.

When the pair $(A_L,\Theta_L)$ is the Jacobian $(J_C,\Theta_C)$ of a hyperelliptic curve $C$ of genus $g\geq 3$, then  $\epsilon(L)\geq 2$ by Theorem \ref{thm:nfoldhype}. When $(A_L,\Theta_L)$ is the Jacobian of a non-hyperlliptic curve of genus $3$, we conclude the same by Theorem \ref{thm:threefoldnohype}. It remains, by the validity of Conjecture \ref{conj:main} for principle polarizations, to tackle the case when $(A_L,\Theta_L)$ is decomposable, meaning there is an isomorphism
	\[
	(A_L,\Theta_L)\ \simeq \ (A^L_1,\Theta^L_1)\ \times\ \ldots \ \times   (A^L_r,\Theta^L_r)\ ,
	\]
	for some positive integer $r\geq 2$, where each $\Theta^L_i$ is an irreducible theta divisor.
	
In what follows fix $F\subseteq A$ an irreducible curve. Under our assumptions, note that for each $i=1,\ldots ,r$ and any smooth point $x_i\in \textup{Supp}(\Theta^L_i)$, we have the following numerical equality
	\[
	L \ \equiv_{\textup{num}} \ \sum_{i=1}^{i=r}\overbrace{f^*\big(A^L_1\times \ldots A^L_{i-1}\times(\Theta^L_i-x_i)\times A^L_{i+1}\times \ldots \times A^L_r\big)}^{D^L(x_i)} \ .
	\]
	In particular, if there exists two smooth points $x_i\in \textup{Supp}(\Theta^L_i)$ and $x_j\in \textup{Supp}(\Theta^L_j)$, so that 
	\[
	F\nsubseteq \textup{Supp}(D^L(x_i))  \textup{ and } F\nsubseteq \textup{Supp}(D^L(x_j)) \ ,
	\]
	then  B\'ezout's theorem leads to the following inequality
	\[
	(L\cdot F)\ \geq \ \mult_{0}(D^L_{i})\mult_{0}(F) \ + \ \mult_{0}(D^L_{i})\mult_{0}(F) \ = \ 2\mult_{0}(F) \ .
	\]
So, in order to show that $\epsilon(L)\geq 2$, it remains to see what happens for those curves $F$ that are contained in the support of at least $r-1$ of these divisors no matter which smooth points we are taking on the respective theta divisors. Without loss of generality we can assume that 
	\[
	F  \ \subseteq \textup{Supp}\Big(f^*\big(A^L_1\times \ldots A^L_{i-1}\times(\Theta^L_i-x_i)\times A^L_{i+1}\times\ldots \times A^L_r\big)\Big), 
	\]
	for any smooth point $x_i\in\textup{Supp}(\Theta^L_i)$ and each $i=1,\ldots ,r-1$.
	
As $\Theta^L_i$ is irreducible, then \cite[Proposition~4.4.1]{BL04} yields that the image of the Gau\ss \ map defined by it is not contained in a hyperplane. In particular, by semi-continuity this implies that 
	\[
	\bigcap_{x_i\in \textup{Supp}(\Theta^L_i) \textup{-smooth pt.}} \textup{Supp}(\Theta^L_i-x_i)\ = \ 0_{A^L_i}\ .
	\] 
	So, going back to our curve $F$, then this equality forces the following inclusion
	\[
	F \ \subseteq \ f^{-1}\big(\overbrace{\{0_{A^L_1}\}\times \ldots \ \{0_{A^L_{r-1}}\}\times A^L_r}^{A_L'}\big) \ .
	\]
As $A_L'$ is abelian, $f^{-1}(A_L')$ remains abelian. Let $A'$ be the component, containing $F$. Its dimension is less than $g$ and for the restriction map $f|_{A'} : A'\rightarrow A_L'$ we have $L|_{A'} =(f_{A'})^*(\Theta_L')$. Now, $L|_{A'}$ cannot be principle by \cite[Lemma~1]{DH07}, as $(A,L)$ is indecomposable. So, our inductive assumptions force $\epsilon(L|_{A'})\geq 2$ and we get $(L\cdot F)\geq 2\cdot \mult_0(F)$. Apllying these ideas to all the curve on $A$, the definition of Seshadri constants implies that $\epsilon(L)\geq 2$, also when $(A_L,\Theta_L)$ is decomposable.

\end{proof}

\begin{proof}[Proof of Corollary \ref{cor:main}]
When $g=3$ we know that a principle polarized abelian threefold is always a Jacobian of a smooth curve or is decomposable. Moreover, by \cite[Theorem 1]{BS01} or \cite[Theorem 7]{D04} we know the Seshadri constants as discussed in Section $3$ and Section $4$. In particular, then we obtain the statement if we take into account Theorem \ref{thm:main}.

When $g=4$, the same ideas hold if we know that an irreducible principle polarized abelian fourfold $(A,\Theta)$ with $\epsilon(\Theta)<2$ has to be a Jacobian of a hyper-elliptic curve. If $(A,\Theta)$ is not a Jacobian then by \cite[Theorem 7]{I95} we know
\[
\{x\in \textup{Supp}(D) \ | \ D\in |2\Theta|, \mult_0(D)\geq 4\} \ = \ \{0\}
\]
and hence $\epsilon(\Theta)\geq 2$ in this case. And when $(A,\Theta)$ is a Jacobian of a non-hyperelliptic curve, then $\epsilon(\Theta)=2$ by \cite[Theorem 7]{D04}. This finishes the proof.
\end{proof}

\section{Appendix: multiplicity under \'etale maps}

The goal of this section is to describe o formula for the behaviour under \'etale maps of the multiplicity at a point of a subvariety. This formula most surely is known to the experts but for completeness we include here its proof.
\begin{proposition}[Multiplicity under \'etale maps]\label{prop:etale}
	Let $f:X\rightarrow Y$ be an \'etale map of degree $d\geq 2$ between two smooth varieties. Let $V\subseteq X$ an irreducible subvariety passing through the point $x_1\in X$, and let $x_1,\ldots x_d$ be the points in the fiber $f^{-1}(f(x))$. Then we have the following inequality
	\[
	\textup{deg}\big(V\rightarrow f(V)\big)\cdot \mult_{f(x)}(f(V))\ = \ \sum_{i=1}^{i=d}\cdot \mult_{x_i}(V) \ .
	\]
	
\end{proposition}
\begin{proof}
	Let $W\subseteq Y$ be a subvariety passing through some point $y\in Y$. Denote by $\pi_Y:\overline{Y}\rightarrow Y$ the blow-up of $Y$ at the point $y$ with $E_Y$ the exceptional divisor and by $\overline{W}$ be the proper transform of $W$ through the blow-up map $\pi_Y$. Under this notation we can translate the multiplicity as an intersection number on the blow-up borrow the following \cite[p.79]{F84}:
	\begin{equation}\label{eq:mult}
	\mult_y(W) \ =\ -\overline{W}\cdot (-E_Y)^{\textup{dim}(W)}\ .
	\end{equation}
	Taking this into account, denote by $y=f(x)$ and let $W=f(V)$. With this in hand we will be using the following commutative diagram:
	\[ \begin{tikzcd}
	\overline{X} \arrow{r}{\overline{f}} \arrow[swap]{d}{\pi_X} & \overline{Y} \arrow{d}{\pi_Y} \\%
	X \arrow{r}{f}& Y
	\end{tikzcd}
	\]
	Here the map $\pi_X$ is the blow-up of $X$ at the points $x_1,\ldots ,x_d$ with the exceptional divisors $E_1,\ldots ,E_d$. Due to the fact that $f$ is an \'etale map of degree $d$, then we know for sure that
	\[
	\overline{f}^*(E_Y)\ = \ E_1\ + \ \ldots\  +\ E_d \ .
	\]
	Let $\overline{W}\subseteq \overline{Y}$ be the proper transform of $W$ through $\pi_Y$ and $\overline{V}$ of $V$ through $\pi_X$. 
	
	With these in hand we proceed to prove the main equality. Since no two divisors from $E_1,\ldots ,E_d$ intersect and making use of $(\ref{eq:mult})$ , we then deduce the following list of equalities
	\[
	\big(-\overline{V}\cdot \big(-\overline{f}^*(E)\big)^{\textup{dim}(\overline{V})} \big) \ = \ \big(-\overline{V}\cdot \big(-\sum_{i=1}^{i=d} E_i\big)^{\textup{dim}(\overline{V})} \big) \ =\ \sum_{i=1}^{i=d}\big(-\overline{V}\cdot \big(-E_i)\big)^{\textup{dim}(\overline{V})} \big) \ = \ \sum_{i=1}^{i=d} \mult_{x_i}(V) \ .
	\]
	Now projection formula for intersection numbers and the fact that $\overline{Y}=\overline{f}(\overline{V})$ yield also 
	\[
	\big(-\overline{V}\cdot \big(-\overline{f}^*(E)\big)^{\textup{dim}(\overline{V})} \big) \ = \ \textup{deg}\big(\overline{V}\rightarrow \overline{W}\big)\cdot \big(-\overline{W}\cdot \big(-E\big)^{\textup{dim}(\overline{W})} \big) \ = \ \textup{deg}\big(V\rightarrow W\big)\cdot \mult_y(W)\ .
	\]
	Finally putting together the last two sequences of equalities implies easily the statement and this finishes the proof.
	
\end{proof}

\end{document}